\documentclass{amsart}

\usepackage{amscd,amssymb,amsmath,graphicx,verbatim,xypic}
\usepackage[dvips]{hyperref}
\usepackage[TS1,OT1,T1]{fontenc}
\newtheorem{theorem}{Theorem}[section]
\newtheorem{lemma}[theorem]{Lemma}
\newtheorem{corollary}[theorem]{Corollary}
\newtheorem{proposition}[theorem]{Proposition}

\theoremstyle{definition}
\newtheorem{definition}[theorem]{Definition}
\newtheorem{example}[theorem]{Example}

\theoremstyle{remark}
\newtheorem{remark}[theorem]{Remark}

\DeclareMathOperator{\Img}{Im} \DeclareMathOperator{\Hom}{Hom}
\DeclareMathOperator{\Ext}{Ext} \DeclareMathOperator{\End}{End}
\DeclareMathOperator{\tria}{tria} \DeclareMathOperator{\coh}{coh}
\DeclareMathOperator{\fdmod}{mod} \DeclareMathOperator{\Ind}{Ind}
\DeclareMathOperator{\Aut}{Aut}
\DeclareMathOperator{\Ker}{Ker}

\newcommand{\Dcal}{\ensuremath{\mathcal{D}}}
\newcommand{\Xcal}{\ensuremath{\mathcal{X}}}
\newcommand{\Ycal}{\ensuremath{\mathcal{Y}}}
\newcommand{\Tcal}{\ensuremath{\mathcal{T}}}
\newcommand{\Acal}{\ensuremath{\mathcal{A}}}
\newcommand{\Hcal}{\ensuremath{\mathcal{H}}}
\newcommand{\Rcal}{\ensuremath{\mathcal{R}}}

\newcommand{\D}{\mathbb{D}}
\newcommand{\K}{\mathbb{K}}
\newcommand{\Z}{\mathbb{Z}}
\newcommand{\N}{\mathbb{N}}

\newcommand{\ra}{\rightarrow}

\numberwithin{equation}{section}

\begin{document}
\title{{t}-structures via recollements for piecewise hereditary algebras}
\author{Qunhua Liu, Jorge Vit{\'o}ria}\thanks{The second named author is supported by DFG - SPP Darstellungstheorie 1388}
\address{Qunhua Liu, Jorge Vit{\'o}ria, Institute for algebra and number theory, University of Stuttgart, Pfaffenwaldring 57, D-70569 Stuttgart, Germany}
\email{qliu@mathematik.uni-stuttgart.de, vitoria@mathematik.uni-stuttgart.de}
\date{\today}

\begin{abstract}
Following the work of Beilinson, Bernstein and Deligne, we study
restriction and induction of t-structures in triangulated categories
with respect to recollements. For derived categories of piecewise
hereditary algebras we give a necessary and sufficient condition for
a bounded t-structure to be induced from a recollement by derived
categories of algebras. As a corollary we prove that for
hereditary algebras of finite representation type all bounded
t-structures can be obtained in this way. \medskip 
\end{abstract}
\maketitle

\begin{section}{Introduction}
The concepts of recollement and t-structure go back to the work of
Beilinson, Bernstein and Deligne (\cite{BBD}).  In that paper, the
interplay between these two ideas is explored in order to build the
category of perverse sheaves. In fact, this category is seen as the
heart of some t-structure induced from a recollement of the derived
category of constructible sheaves on a stratified scheme. This
technique of building a t-structure out of a recollement is the
major topic of their paper.

One of the main results of \cite{BBD} states that the heart of a
t-structure is an abelian category. Of course every abelian category
can be seen as the heart of a t-structure in its derived category
(the standard t-structure) but the question of which abelian
categories can be found as hearts of t-structures in a fixed derived
category remains as an interesting research topic. In 1996, Happel,
Reiten and Smal{\o} (\cite{HRS}) studied the interactions between
t-structures and tilting theory, turning this concept into an
important notion in the representation theory of finite dimensional
algebras. Recently, interesting interactions between bounded
t-structures for finite dimensional algebras and certain sets of
\textit{simple-minded} objects has appeared in the work of Keller
and Nicol{\'a}s (\cite{KN}) and Koenig and Yang (\cite{KY}). In their
work, sets of simple-minded objects are used to construct
equivalences of derived categories sending these sets to sets of
simple objects.

The full set of bounded t-structures on a triangulated category is
known only for a very few examples. It is known for division rings
and semisimple algebras, although this seems to not yet be written. Thus, for the convenience of the reader, we include them
in section 4. A famous (nontrivial) example for which much is known about t-structures
is that of the Kroenecker quiver. In \cite{GKR} one can find a
complete list of bounded t-structures, up to autoequivalence, on the derived category of coherent
sheaves over the projective line, which is well known to be
equivalent to the derived category of finitely generated modules
over the Kroenecker quiver. In our last section (Example \ref{coh}) we observe that the
t-structure with heart coh$(\mathbb{P}^1)$ in this triangulated category cannot be induced from a recollement of derived categories.
Moreover, the classification of bounded t-structures obtained in \cite{GKR} for
the Kroenecker quiver is a consequence of the study of the space of
stability conditions on the projective line. These ideas with origin
in the work of Bridgeland (\cite{Br}) are of significant geometric
interest and have been intensively studied for the past years.

On the other hand, recollements for derived categories and
connections to tilting theory had been explored in the work of
Parshall and Scott (\cite{PS}). More recently, Angeleri H{\"u}gel,
Koenig and the first named author (\cite{ALK1}) have approached tilting theory using the general theory of recollements for triangulated categories. In that paper the authors provide ways of
constructing tilting objects from recollements and constructing recollements from tilting modules. Our approach to
this paper is of a similar spirit but with t-structures in mind
instead of tilting objects. Also, in \cite{ALK2}, Angeleri
H{\"u}gel, Koenig and the first named author  proceed to prove a
Jordan H\"older theorem for derived categories of hereditary artin algebras, later generalised  to piecewise hereditary algebras (\cite{AKL3}).

In this paper we study the interaction between these two key
concepts: recollements and t-structures. We use the construction set by Beilinson, Bernstein and
Deligne (\cite{BBD}) to study properties of t-structures in the derived categories
of piecewise hereditary algebras, using for that purpose some of the
examples and techniques from \cite{ALK2, AKL3}, namely how to construct a recollement from an indecomposable and exceptional (i.e. without self-extensions) object. Piecewise hereditary algebras were defined in \cite{H} and \cite{HRS1}. An algebra is piecewise hereditary if it is derived equivalent to a certain hereditary abelian category. This covers, for example, hereditary algebras, (quasi-)tilted algebras and canonical algebras. Our main result states that for a piecewise hereditary algebra, the bounded t-structures whose heart is a length category can be obtained by a BBD-induction with respect to some recollement.

The paper is structured as follows. In section
2 we recall some preliminaries on recollements and t-structures that will be needed later. In section 3 we recall the definition of BBD-induction and restriction of t-structures with respect to a recollement and
investigate general properties of an induced t-structure, namely nondegeneracy, boundedness (see also \cite{CT,W}) and when it possesses a heart of finite length. Also, it is observed that, when restriction is possible, it is an inverse process to induction.  In section 4 we discuss t-structures for semisimple
algebras and in section 5 we prove that all bounded t-structures for
$A_n$ are induced with respect to a recollement associated with an
idempotent. In the case $n=2$ we can actually describe explicitly
all the possible t-structures. Finally, section 6 proves the main
result, using the techniques introduced in \cite{KN, KY}.
\end{section}

\begin{section}{Preliminaries}

In this section we recall the definition and some properties of recollements, (bounded)
t-structures and simple-minded objects in triangulated categories.

\begin{subsection}{Recollements} Let $\Xcal, \Ycal, \Dcal$ be triangulated categories. $\Dcal$ is said to be  a
{\em recollement} of $\Xcal$ and $\Ycal$ if there are six triangle
functors as in the following diagram
\begin{equation}\nonumber
\begin{xymatrix}{\mathcal{Y}\ar[r]^{i_*=i_!}&\mathcal{D}\ar@<3ex>[l]_{i^!}\ar@<-3ex>[l]_{i^*}\ar[r]^{j^*=j^!}&\mathcal{\mathcal{X}}\ar@<3ex>_{j_*}[l]\ar@<-3ex>_{j_!}[l]}.
\end{xymatrix}
\end{equation}
such that
\begin{enumerate}
\item $(i^\ast,i_\ast)$,\,$(i_!,i^!)$,\,$(j_!,j^!)$ ,\,$(j^\ast,j_\ast)$
are adjoint pairs;

\item $i_\ast,\,j_\ast,\,j_!$  are full embeddings;

\item  $i^!\circ j_\ast=0$ (and thus also $j^!\circ i_!=0$ and
$i^\ast\circ j_!=0$);

\item for each $Z\in \Dcal$ there are triangles \[i_! i^!Z\to
Z\to j_\ast j^\ast Z\to i_! i^!Z[1]\]
\[j_! j^! Z\to Z\to
i_\ast i^\ast Z\to j_!j^!Z[1].\]
\end{enumerate}

The following lemma follows easily from the definition of
recollement and will be needed later.
\begin{lemma} \label{newreco}
Let $\mathcal{D}$ be a triangulated category and $\Phi$ be an
autoequivalence of $\mathcal{D}$, with $\Psi$ its quasi-inverse. If
\begin{equation}\nonumber
\begin{xymatrix}{\mathcal{Y}\ar[r]^{i_*}&\mathcal{D}\ar@<3ex>[l]_{i^!}\ar@<-3ex>[l]_{i^*}\ar[r]^{j^*}&\mathcal{\mathcal{X}}\ar@<3ex>_{j_*}[l]\ar@<-3ex>_{j_!}[l]}
\end{xymatrix}
\end{equation}
is a recollement, then so is
\begin{equation}\nonumber
\begin{xymatrix}{\mathcal{Y}\ar[r]^{\Phi i_*}&\mathcal{D}\ar@<3ex>[l]_{i^!\Psi}\ar@<-3ex>[l]_{i^*\Psi}\ar[r]^{j^*\Psi}&\mathcal{\mathcal{X}}\ar@<3ex>_{\Phi j_*}[l]\ar@<-3ex>_{\Phi j_!}[l]}.
\end{xymatrix}
\end{equation}
\end{lemma}
\begin{proof}
$(\Psi, \Phi, \Psi)$ is an adjoint triple and hence axiom (1) of recollement follows. Also, since $\Phi$ and $\Psi$ are fully faithful, axiom (2) is automatically true as well. Now, by definition of quasi-inverse, $\Psi\Phi$ and $\Phi\Psi$ are naturally equivalent to the identity, implying that $i^!\Psi\Phi j_*$ is naturally equivalent to $i^!j_*=0$, thus axiom (3). 
The same argument applies to check axiom (4), hence finishing the proof.
\end{proof}
\end{subsection}

\begin{subsection}{t-structures} For reference, see for example
\cite{BBD,M}. A {\em t-structure} in a triangulated
category $\Dcal$ is a pair $(\Dcal^{\leq 0}, \Dcal^{\geq 0})$ of
strictly full subcateogies with the following properties: write
$\Dcal^{\leq n} = \Dcal^{\leq 0}[-n]$ and $\Dcal^{\geq n} =
\Dcal^{\geq 0}[-n]$ for $n\in \Z$,
\begin{enumerate}
\item $\Dcal^{\leq 0} \subset \Dcal^{\leq 1}$, $\Dcal^{\geq 1} \subset \Dcal^{\geq
0}$;

\item $\Hom_{\Dcal}(X,Y) = 0$ for all $X\in \Dcal^{\leq 0}$ and
$Y\in\Dcal^{\geq 1}$;

\item for each $Z$ in $\Dcal$ there is a distinguished triangle $X \ra Z \ra Y \ra
X[1]$ where $X\in\Dcal^{\leq 0}$ and $Y\in\Dcal^{\geq 1}$.
\end{enumerate}
By definition, $\Dcal^{\leq 0} = \{X \in \Dcal: \Hom_{\Dcal}(X,Y) =
0,\ \forall\ Y\in \Dcal^{\geq 1} \}$ and $\Dcal^{\geq 1} = \{Y \in
\Dcal: \Hom_{\Dcal}(X,Y) = 0,\ \forall\ X\in \Dcal^{\leq 0} \}$. The
subcategory $\Dcal^{\leq 0}$ is called an {\em aisle}, and
$\Dcal^{\geq 0}$ is called a {\em coaisle}. The {\em heart}
$\Dcal^{\leq 0} \cap \Dcal^{\geq 0}$ is always an abelian category.
This is a full subcategory of $\Dcal$. But in general its derived
category has nothing to do with the original triangulated category
$\Dcal$. Nevertheless the first extension group is preserved.

\begin{lemma}[\cite{M} Chapter 4, 2.1.1] Let $\Dcal$ be a triangulated
category with a t-structure $(\Dcal^{\leq 0},\Dcal^{\geq 0})$, and
$\Acal = \Dcal^{\leq 0} \cap \Dcal^{\geq 0}$ be the heart. For any
objects $X$, $Y$ in $\Acal$ it holds that
$\Ext^1_{\Acal}(X,Y)=\Hom_{\Dcal}(X,Y[1])$.
\label{extension-in-heart}
\end{lemma}

For a given object $Z\in\Dcal$, the canonical triangle as in $(3)$
is unique (up to equivalence). The associated objects $X$ and $Y$
define the {\em truncation functors} $\tau_{\leq 0}: \Dcal \ra
\Dcal^{\leq 0}$ and $\tau_{\geq 1}: \Dcal \ra \Dcal^{\geq 1}$. For
$n\in\Z$, define $\tau_{\leq n}: \Dcal \ra \Dcal^{\leq n}$ to be
$[-n]\tau_{\leq 0}[n]$ and $\tau_{\geq n}:\Dcal \ra \Dcal^{\geq n}$
to be $[-n+1]\tau_{\geq 1}[n-1]$. The composition $\tau_{\geq
n}\tau_{\leq n}$ provides the {\em cohomological functor} $H^n$ from
$\Dcal$ to $\Dcal^{\leq n}\cap\Dcal^{\geq n} = (\Dcal^{\leq
0}\cap\Dcal^{\geq 0})[-n]$ and it can be shown to be equal to $H^n=[-n]H^0[n]$.

Let $(\Dcal^{\leq 0},\Dcal^{\geq 0})$ and $(\tilde{\Dcal}^{\leq 0},\tilde{\Dcal}^{\geq 0})$ be, respectively, t-structures in $\Dcal$ and $\tilde{\Dcal}$, triangulated categories. An exact functor $F: \Dcal \rightarrow \tilde{\Dcal}$ is said to be {\em right} (respectively {\em left}) {\em t-exact} with respect to these t-structures if $F(\Dcal^{\leq 0})\subset\tilde{\Dcal}^{\leq 0}$ (respectively, if $F(\Dcal^{\geq 0})\subset\tilde{\Dcal}^{\geq 0}$).  The functor $F$ is said to be {\em t-exact} if it is both left and right t-exact.
\end{subsection}

\begin{subsection}{Nondegeneracy and boundedness}
A t-structure $(\Dcal^{\leq 0}, \Dcal^{\geq 0})$ in a triangulated
category $\Dcal$ is said to be {\em nondegenerate} if
\begin{equation}\nonumber
\bigcap_{n\in\mathbb{Z}}\mathcal{D}^{\leq n} = 0\ \  {\rm and}
\bigcap_{n\in\mathbb{Z}}\mathcal{D}^{\geq n} = 0
\end{equation}
and it is {\em bounded} if
\begin{equation}\nonumber
\bigcup_{n\in\mathbb{Z}}\mathcal{D}^{\leq n} = \mathcal{D}\ \  {\rm
and} \bigcup_{n\in\mathbb{Z}}\mathcal{D}^{\geq n} = \mathcal{D}.
\end{equation}
By \cite[Chapter 4, 1.3.1 and 1.3.4]{M}, a t-structure is nondegenerate if and only if every
object in $\Dcal$ with zero cohomology in every degree is isomorphic to the zero
object, and it is bounded if and only if it is nondegenerate and every object in $\Dcal$ has bounded
cohomology.

If a t-structure  $(\Dcal^{\leq 0}, \Dcal^{\geq 0})$ is nondegenerate and the functors $H^n$, with $n\in\mathbb{Z}$, are the cohomological functors associated with it, it is easy to see that (check \cite{M} for details):
\[\Dcal^{\leq n} = \{X\in\Dcal: H^i(X) = 0, \ \forall\ i>n\}\]
\[\Dcal^{\geq n} = \{X\in\Dcal: H^i(X)=0, \ \forall\ i<n\}.\]
In the derived category $\Dcal(\Acal)$ of an abelian category $\Acal$ the standard t-structure is nondegenerate (and bounded if considered in $\Dcal^b(\Acal)$) and we denote it by $(\Dcal_0^{\leq 0}, \Dcal_0^{\geq 0})$.

For a class of objects $\mathcal{S}\subset \Dcal$ we denote by $\tria\mathcal{S}$ the smallest strictly full triangulated subcategory of $\Dcal$ containing $\mathcal{S}$.

\begin{lemma}\label{grothendieckgroup} Let $(\Dcal^{\leq 0}, \Dcal^{\geq 0})$  be
a t-structure in a triangulated category $\Dcal$, and $\Acal=\Dcal^{\leq 0}\cap\Dcal^{\geq 0}$ be the heart. If the t-structure is bounded, then $\tria \Acal = \Dcal$. In particular, the Grothendieck group of $\Acal$ coincides with that of $\Dcal$.
\end{lemma}

\begin{proof} Suppose $(\Dcal^{\leq 0}, \Dcal^{\geq 0})$ is a bounded t-structure in
$\Dcal$. So every object $X$ in $\Dcal$ has bounded cohomology (with
respect to the t-structure). Using the truncation functors and the
associated canonical triangles, we see that $X$ is generated by its
cohomology in finitely many steps. The statement follows.
\end{proof}

\end{subsection}

\begin{subsection}{Simple-minded objects}\label{simple-minded}
Let $A$ be a
finite dimensional $\mathbb{K}$-algebra. The bounded derived category of finite dimensional right $A$-modules will throughout be denoted by $\Dcal^b(A)$.

\begin{definition}
A set $X_1,...,X_n$ of objects in $\mathcal{D}^b(A)$ is a {\em
family of simple-minded objects} if the following conditions hold:
\begin{enumerate}
\item $\Hom(X_i, X_j[m])=0$, for all $m<0$;
\item $\Hom(X_i,X_j)=\delta_{ij}\mathbb{K}$, where $\delta_{ij}$ is the Kroenecker delta;
\item The set generates $\mathcal{D}^b(A)$, i.e. $\tria(X_1\oplus\ldots \oplus X_n) = \Dcal^b(A)$.
\end{enumerate}
Two families of simple-minded objects are called {\em equivalent} if
they have the same closure under extensions.
\end{definition}

Simple-minded objects play an important role when studying bounded t-structures of the derived category. Indeed the following holds.

\begin{theorem}[\cite{KY} Corollary 3.9; \cite{KN} Corollary 11.5]\label{KY-thm} There is a bijection between the set of bounded t-structures
in $\Dcal^b(A)$ whose heart is a length category (i.e., every object
has finite length) and the set of
equivalence classes of families of simple-minded objects of $\Dcal^b(A)$.
In particular, the heart of a bounded t-structure in $\Dcal^b(A)$ is
a length category if and only if it is equivalent to a module
category of some finite dimensional algebra. \label{bijection}
\end{theorem}

Given a bounded t-structure in $\Dcal^b(A)$, if the heart is a
length category, it is then equivalent to mod$(\Gamma)$ for some
finite dimensional algebra $\Gamma$ (which is constructed as the
zero-th cohomology of the dg algebra $\tilde{\Gamma}$ in \cite{KY}).
The correspondence in the theorem can be made explicit by
associating to the bounded t-structure the set of simple
$\Gamma$-modules. By Lemma \ref{grothendieckgroup}, the rank of the
Grothendieck group of the heart, which is the number of the simple
$\Gamma$-modules, equals the rank of the Grothendieck group of
$\Dcal^b(A)$, and also of the algebra $A$. Conversely, given a
family of simple-minded objects $X_1,\ldots, X_n$, define the aisle
(coaisle) to be the extension closure of all non-negative
(non-positive) shifts of $X_i$'s. The extension closure of the $X_i$'s
is just the heart of the associated t-structure. In particular the
number $n$ equals the rank of the Grothendieck group of $A$.

We also have the following useful lemma.

\begin{lemma}\label{contains standard aisle}
Suppose $\Acal$ is a length abelian category with finitely many simple objects $X_1,...,X_n$. Let $T=(\mathcal{D}^{\leq
0},\mathcal{D}^{\geq 0})$ be a  bounded t-structure in $\mathcal{D}^b(\Acal)$.
Then $\mathcal{D}^{\leq 0}$ contains a shift of the standard
aisle.
\end{lemma}
\begin{proof}
Since the t-structure is bounded, for each $1\leq i\leq n$, there is an integer $k_i$ such that $X_i\in\Dcal^{\leq k_i}$. Define
\begin{equation}\nonumber
k:={\rm min}\left\{k_i: 1\leq i\leq n\right\}
\end{equation}
and then it is clear that $X_1,...,X_n\in\Dcal^{\leq k}$. Since $\Acal$ is the extension closure of $X_1,...,X_n$ and an aisle is closed under extensions, we have that $\Acal\subset \Dcal^{\leq k}$. By definition of t-structure, if $X\in\Dcal^{\leq n}$ for some $n\in\mathbb{Z}$ then $X[i]\in\Dcal^{\leq n}$ for all $i\geq 0$, and thus $\Dcal_0^{\leq 0}\subset \Dcal^{\leq k}$, or equivalently, $\Dcal_0^{\leq -k}\subset \Dcal^{\leq 0}$.
\end{proof}

\end{subsection}

\end{section}

\begin{section}{BBD-induction and BBD-Restriction}

Throughout, $\mathcal{D}$ will be a triangulated category. We denote the set of all t-structures in $\mathcal{D}$ by $\mathcal{T}_\mathcal{D}$. We shall assume, unless otherwise stated, that this triangulated category has a recollement, denoted by $\mathcal{R}$, by triangulated categories $\mathcal{\mathcal{X}}$ and $\mathcal{Y}$ as follows:
\begin{equation}\nonumber
\begin{xymatrix}{\mathcal{Y}\ar[r]^{i_*}&\mathcal{D}\ar@<3ex>[l]_{i^!}\ar@<-3ex>[l]_{i^*}\ar[r]^{j^*}&\mathcal{\mathcal{X}}\ar@<3ex>_{j_*}[l]\ar@<-3ex>_{j_!}[l]}.
\end{xymatrix}
\end{equation}

The theorem that motivates our approach is the following.

\begin{theorem}[\cite{BBD} Theorem 1.4.10]
Suppose $T_\mathcal{\mathcal{X}}=(\mathcal{\mathcal{X}}^{\leq 0},\mathcal{\mathcal{X}}^{\geq 0})$ and $T_\mathcal{Y}=(\mathcal{Y}^{\leq 0},\mathcal{Y}^{\geq 0})$ are t-structures in $\mathcal{\mathcal{X}}$ and $\mathcal{Y}$ respectively. Then $T=(\mathcal{D}^{\leq 0},\mathcal{D}^{\geq 0})$ defined by:
\begin{equation}\nonumber
\mathcal{D}^{\leq 0}:=\left\{Z\in \mathcal{D}: j^*Z\in \mathcal{\mathcal{X}}^{\leq0},\ i^*Z\in \mathcal{Y}^{\leq0}\right\}
\end{equation}
\begin{equation}\nonumber
\mathcal{D}^{\geq 0}:=\left\{Z\in \mathcal{D}: j^*Z\in \mathcal{\mathcal{X}}^{\geq0},\ i^!Z\in \mathcal{Y}^{\geq0}\right\}
\end{equation}
is a t-structure in $\mathcal{D}$.
\end{theorem}

We will use the notation $T=:$ Ind$(T_\mathcal{\mathcal{X}},T_\mathcal{Y})$ and call $T$ the {\em BBD-induction} (or simply {\em induction}) of $T_\mathcal{\mathcal{X}}$ and $T_\mathcal{Y}$.

The problem of restricting t-structures is a more delicate issue and it is also discussed in \cite{BBD}. We include the proof for sake of completion.

\begin{proposition}[\cite{BBD} Proposition 1.4.12]\label{BBD compat}
Let $T=(\Dcal^{\leq 0},\Dcal^{\geq 0})$ be a t-structure in $\Dcal$. The following are equivalent:
\begin{enumerate}
\item $j_!j^*$ is right t-exact, i.e., $j_!j^*\Dcal^{\leq 0}\subset \Dcal^{\leq 0}$;
\item $j_*j^*$ is left t-exact, i.e., $j_*j^*\Dcal^{\geq 0}\subset \Dcal^{\geq 0}$;
\item $T$ is induced with respect to the recollement $\Rcal$.
\end{enumerate}
\end{proposition}
\begin{proof}
The fact that the first two statements are equivalent since $(j_!j^*, j_*j^*)$ is an adjoint pair.  Note that these two conditions imply that $i_*i^*$ is right t-exact and $i_*i^!$ is left t-exact by the canonical triangles of the recollement. This implies, in particular, that Hom$(i^*\Dcal^{\leq 0},i^!\Dcal^{\geq 1})=0$.

It is clear that (3) implies (1) (and thus (2)). Indeed, if  $(\Dcal^{\leq 0},\Dcal^{\geq 0})$ is induced with respect to the recollement, then by definition $j^*\Dcal^{\leq 0}\subset \Xcal^{\leq 0}$. Also, $j_!\Xcal^{\leq 0}\subset \Dcal^{\leq 0}$ since, by adjunction, for $X\in\Xcal^{\leq 0}$ and $Y\in\Dcal^{\geq 1}$ (which implies that $j^*Y\in\Xcal^{\geq 1}$),
\begin{equation}\nonumber
{\rm Hom}_{\Dcal}(j_!X,Y)={\rm Hom}_{\Xcal}(X,j^*Y)= 0.
\end{equation}

Conversely, suppose that (1) (and (2)) hold. Then it is easy to check that $(j^*\Dcal^{\leq0},j^*\Dcal^{\geq 0})$ is a t-structure in $\Xcal$ (and thus $j^*$ is t-exact). Observe that  $i_*^{-1}(i_*\Ycal \cap \Dcal^{\leq 0})=i^*\Dcal^{\leq 0}$ and $i_*^{-1}(i_*\Ycal \cap \Dcal^{\geq 0})=i^!\Dcal^{\geq 0}$, since $i_*i^*$ is right t-exact and $i_*i^!$ is left t-exact. We check now that the pair $(i_*^{-1}(i_*\Ycal \cap \Dcal^{\leq 0}),i_*^{-1}(i_*\Ycal \cap \Dcal^{\geq 0}))$ is a t-structure in $\Ycal$. Since $i_*$ is a full embedding, it is enough to check that $(i_*\Ycal \cap \Dcal^{\leq 0},i_*\Ycal \cap \Dcal^{\geq 0})$ is a t-structure in $i_*\Ycal$. The orthogonality condition follows from the observations above. We only need to check the triangle condition of a t-structure. Let $Y\in\Ycal$. Since $j^*$ is t-exact then $j^*\tau^{\leq0}i_*Y=\tau_\Xcal^{\leq 0}j^*i_*Y=0$ and similarly, $j^*\tau^{\geq 1}i_*Y=0$. Thus the truncations $\tau^{\leq0}i_*Y$ and $\tau^{\geq 1}i_*Y$ lie in the image of $i_*$ and thus we get a triangle
\begin{equation}\nonumber
\tau^{\leq 0}i_*Y\rightarrow i_*Y\rightarrow \tau^{\geq 1}i_*Y\rightarrow \tau^{\leq 0}i_*X[1]
\end{equation}
in which $\tau^{\leq 0}i_*Y\in i_*\Ycal \cap \Dcal^{\leq 0}$ and $\tau^{\geq 1}i_*Y\in i_*\Ycal \cap \Dcal^{\geq 1}$.

It remains to check that $(\Dcal^{\leq 0},\Dcal^{\geq 0})$ is BBD-induced with respect to $\Rcal$. We prove that indeed it is induced with respect to $\Rcal$ by taking $T_\Xcal:=(j^*\Dcal^{\leq0},j^*\Dcal^{\geq 0})$ in $\Xcal$ and $T_\Ycal:=(i^*\Dcal^{\leq 0}, i^!\Dcal^{\geq 0})$ in $\Ycal$. Let $\tilde{T}={\rm Ind}(T_\Xcal,T_\Ycal)$. Clearly $\tilde{T}$ is determined by the aisle
\begin{equation}\nonumber
\tilde{\mathcal{D}}^{\leq 0}:=\left\{Z\in \mathcal{D}: j^*Z\in j^*\mathcal{D}^{\leq0},\ i^*Z\in i^*\mathcal{D}^{\leq0}\right\}
\end{equation}\nonumber
and it is clear that $\mathcal{D}^{\leq 0}\subset\tilde{\mathcal{D}}^{\leq 0}$. We prove that $\tilde{\mathcal{D}}^{\leq 0}=\mathcal{D}^{\leq 0}$, thus proving that $\tilde{T}=(\Dcal^{\leq 0},\Dcal^{\geq 0})$. If $Z\in \tilde{\mathcal{D}}^{\leq 0}$, $j_!j^*Z\in \mathcal{D}^{\leq 0}$ and $i_*i^*Z\in\mathcal{D}^{\leq 0}$ since $j_!$ and $i_*$ are  right t-exact with respect to $(\Dcal^{\leq 0},\Dcal^{\geq 0})$, $(j^*\Dcal^{\leq 0},j^*\Dcal^{\geq 0})$ and $(i^*\Dcal^{\leq 0},i^!\Dcal^{\geq 0})$. Thus, the triangle
\begin{equation}\nonumber
j_!j^!Z\longrightarrow Z \longrightarrow i_*i^*Z \longrightarrow
j_!j^!Z[1],
\end{equation}
given by the recollement, proves that $Z\in \mathcal{D}^{\leq 0}$.
\end{proof}

\begin{definition}
Fix a recollement $\mathcal{R}$ for $\mathcal{D}$ as above. A t-structure $T=(\mathcal{D}^{\leq 0},\mathcal{D}^{\geq 0})$ in $\mathcal{D}$ is said to be {\em compatible with the recollement} if $j_!j^*$ is right t-exact. We denote by $\Tcal^{\Rcal}_{\Dcal}$ the set of t-structure in $\mathcal{D}$ which are compatible with the recollement $\mathcal{R}$. For $T=(\Dcal^{\leq 0},\Dcal^{\geq 0})\in\Tcal^{\Rcal}_{\Dcal}$ we define the {\em BBD-restriction} (or simply \textit{restriction}) of $T$ with respect to $\Rcal$ to be the pair Res$(T):=(T_\Xcal,T_\Ycal)\in  \mathcal{T}_\mathcal{\mathcal{X}}\times \mathcal{T}_\mathcal{Y}$ where $T_\Xcal:=(j^*\Dcal^{\leq0},j^*\Dcal^{\geq 0})$ and $T_\Ycal:=(i^*\Dcal^{\leq 0}, i^*\Dcal^{\geq 0})$.
\end{definition}

It follows from the above proof that the aisle and the coisle of $T_{\Ycal}$ are respectively given by
\begin{equation}\nonumber
\begin{array}{c}i^*\Dcal^{\leq 0}=i_*^{-1}(i_*\Ycal\cap\Dcal^{\leq 0}) = \{Y\in\Ycal:i_*Y\in\Dcal^{\leq 0}\}, \\ i^!\Dcal^{\geq 0} = i_*^{-1}(i_*\Ycal\cap\Dcal^{\geq 0}) = \{Y\in\Ycal:i_*Y\in\Dcal^{\geq 0}\}.
\end{array}
\end{equation}
Also if $T_\Ycal$, $T$ and $T_\Xcal$ are compatible t-structures with respect to the recollement $\Rcal$, then the functors $i^*$ and $j_!$ are right t-exact, $i_*$ and $j^*$ are t-exact, and $i^!$ and $j_*$ are left t-exact.

The proposition above provides us with the following two maps
\begin{equation}\nonumber
{\rm Res}: \Tcal^{\Rcal}_{\Dcal}\longrightarrow
\mathcal{T}_\mathcal{\mathcal{X}} \times \mathcal{T}_\mathcal{Y}
\end{equation}
\begin{equation}\nonumber
{\rm Ind}:  \mathcal{T}_\mathcal{\mathcal{X}} \times \mathcal{T_\mathcal{Y}}\longrightarrow \mathcal{T}_\mathcal{D}
\end{equation}
and shows that for a t-structure $T\in\Tcal^{\Rcal}_{\Dcal}$, Ind$($Res$(T))=T$. Indeed we have more.

\begin{corollary}\label{Ind-Res}
Let $\mathcal{R}$ be a recollement as before and let $T_\mathcal{X}=(\Xcal^{\leq 0},\Xcal^{\geq 0})$ and $T_\mathcal{Y}=(\Ycal^{\leq 0},\Ycal^{\geq 0})$ be t-structures in $\Xcal$ and $\Ycal$ respectively. We have that ${\rm Ind}(T_\mathcal{X},T_\mathcal{Y})$ is compatible with $\Rcal$ and ${\rm Res}({\rm Ind}(T_\mathcal{X},T_\mathcal{Y}))=(T_\mathcal{X},T_\mathcal{Y})$. Consequently, ${\rm Res}$ and ${\rm Ind}$ are inverse bijections between $\Tcal^{\Rcal}_{\Dcal}$ and $\mathcal{T}_\mathcal{X}\times\mathcal{T}_\mathcal{Y}$.
\end{corollary}
\begin{proof}
Let $T:={\rm Ind}(T_\mathcal{X},T_\mathcal{Y})=(\mathcal{D}^{\leq 0},\mathcal{D}^{\geq 0})$. It is compatible with $\Rcal$ by Proposition \ref{BBD compat}. Let Res$(T)=((\tilde{\Xcal}^{\leq0},\tilde{\Xcal}^{\geq0}),(\tilde{\Ycal}^{\leq0},\tilde{\Ycal}^{\geq0}))$. Clearly,
\begin{equation}\nonumber
\tilde{\mathcal{X}}^{\leq 0}:=j^*(\mathcal{D}^{\leq 0}) \subset \mathcal{X}^{\leq 0}\ \ \ {\rm and} \ \ \ \tilde{\mathcal{X}}^{\geq 1} :=j^*\mathcal{D}^{\geq 1}  \subset \mathcal{X}^{\geq 1},
\end{equation}
by definition of induced t-structure, and thus $\mathcal{X}^{\leq 0}=\tilde{\mathcal{X}}^{\leq 0}$. On the other hand, let $Y\in \tilde{\mathcal{Y}}^{\leq 0}=\left\{Y\in \mathcal{Y}:
i_*Y\in \mathcal{D}^{\leq 0}\right\}$. Then $Y\cong i^*i_*Y \in
\mathcal{Y}^{\leq 0}$ by definition of $\mathcal{D}^{\leq 0}$. Also, if $Y\in
\mathcal{Y}^{\leq 0}$, then so is $i^*i_*Y$ (because it is
isomorphic to $Y$). Therefore $i_*Y\in \mathcal{D}^{\leq 0}$, which
means that $Y\in \tilde{\mathcal{Y}}^{\leq 0}$.
\end{proof}

This means in particular that once we have a description of all t-structures of the middle term of a recollement, then we have a description of all t-structures of the left and the right hand sides of the recollement.

We now discuss some properties of t-structures that are preserved via induction and restriction. Nondegeneracy is discussed in \cite{BBD} and boundedness is partly discussed in \cite{CT, W}.

\begin{lemma}\label{Ind-res bdd}
The bijections ${\rm Res}$ and ${\rm Ind}$ restrict to bijections between nondegenerate (respectively, bounded) t-structures in $\Tcal^{\Rcal}_{\Dcal}$ and $\mathcal{T}_\mathcal{X}\times\mathcal{T}_\mathcal{Y}$.
\end{lemma}
\begin{proof}
Suppose $(T_\mathcal{X},T_\mathcal{Y})\in\mathcal{T}_\Xcal\times\mathcal{T}_\Ycal$ is a pair of nondegenerate t-structures. Let us check that the nondegeneracy of its BBD-induction $(\Dcal^{\leq0},\Dcal^{\geq0}):=\text{Ind}(T_\Xcal,T_\Ycal)$:
If $Z\in\Dcal^{\leq n}$ (respectively, $Z\in\Dcal^{\geq n}$) for all $n\in\mathbb{Z}$, then $i^*Z\in\Ycal^{\leq n}$ (respectively, $i^!Z\in\Ycal^{\geq n}$) and $j^*Z\in\Xcal^{\leq n}$ (respectively, $j^*Z\in\Xcal^{\geq n}$) for all $n\in\mathbb{Z}$, implying that $i^*Z=j^*Z=0$ (respectively, $i^!Z=j^*Z=0$). By the canonical triangles induced by the recollement we get $Z=0$ and hence the induced t-structure is nondegenerate. Moreover suppose $T_\Xcal$ and $T_\Ycal$ are bounded. For $Z\in\Dcal$, because of the boundedness of $T_\Xcal$ and $T_\Ycal$, there exits integers $m$, $n$, $k$ and $l$ such that $j^*Z\in\Xcal^{\leq m}$ and $\Xcal^{\geq n}$, $i^*Z\in\Ycal^{\leq k}$ and $i^!Z\in\Ycal^{\geq l}$. It follows from the canonical triangles
\[j_!j^!Z\longrightarrow Z \longrightarrow i_*i^*Z \longrightarrow
j_!j^!Z[1]\]
\[i_!i^!Z\longrightarrow Z \longrightarrow j_*j^*Z \longrightarrow
i_!i^!Z[1]\] and from the (left or right) t-exactness of the functors $i_*$, $j_*$ and $j_!$ that $Z\in\Dcal^{\leq \max \left\{ m, k\right\}}$ and $Z\in\Dcal^{\geq \min \left\{n, l\right\}}$. Thus the induced t-structure is bounded as well.

Conversely, let $T=(\Dcal^{\leq0},\Dcal^{\geq0})$ be a nondegenerate t-structure in $\mathcal{D}$ compatible with $\mathcal{R}$. To show the nondegeneracy of the restricting t-structure, we need to show that
$\cap_{n\in\Z}i^*\Dcal^{\leq n} = 0 = \cup_{n\in\Z}i^!\Dcal^{\geq n}$ and $\cup_{n\in\Z}j^*\Dcal^{\leq n} =0 =\cup_{n\in\Z}j^*\Dcal^{\geq n}$.
If $Y\in i^*\Dcal^{\leq n}$ for all $n\in\Z$, then $i_*Y\in i_*i^*\Dcal^{\leq n}$ for all $n\in\Z$. Since $i_*i^*$ is right t-exact, $i_*Y\in\cap_{n\in\Z}\Dcal^{\leq n}=0$. Hence $i_*Y$ and also $Y$ is trivial. The other equalities follows by similar argument.

Moreover suppose now $T$ is bounded. To show the boundedness of restricting t-structures, we need to show $\cup_{n\in\Z}i^*\Dcal^{\leq n} = \Ycal = \cup_{n\in\Z}i^!\Dcal^{\geq n}$ and $\cup_{n\in\Z}j^*\Dcal^{\leq n} =\Xcal =\cup_{n\in\Z}j^*\Dcal^{\geq n}$. For $Y\in\Ycal$, because of the boundedness of $T$, there exists integers $s$ and $t$ such that $i_*Y\in\Dcal^{\leq s}$ and $\Dcal^{\geq t}$. Hence $Y\cong i^*i_*Y$ belongs to $i^*\Dcal^{\leq s}$ and $Y\cong i^!i_*Y$ belongs to $i^!\Dcal^{\geq t}$. Similarly for $X\in\Xcal$, there exists integers $k$ and $l$ such that $j_!X\in\Dcal^{\leq k}$ and $j_*X\in\Dcal^{\geq l}$. It follows that $X\cong j^*j_!X \in j^*\Dcal^{\leq k}$ and $X\cong j^*j_*X\in j^*\Dcal^{\geq l}$.
\end{proof}

There is a natural action of the autoequivalent group $\Aut(\Dcal)$
on the set $\Tcal_{\Dcal}$ of all t-structures in $\Dcal$. Two
t-structures $T_1$ and $T_2$ are called {\em equivalent}, if they
are transformed to each other by an autoequivalence of $\Dcal$, in
other words, the (co-)aisle of $T_1$ is sent to the (co-)aisle of
$T_2$ by some $\Phi\in\Aut(\Dcal)$.

\begin{lemma} Let $T_1$ and $T_2$ be two equivalent t-structures.
Then $T_1$ is a BBD-induction if and only if so is $T_2$.
\label{BBDinduction}
\end{lemma}

\begin{proof} Suppose $\Phi\in\Aut(\Dcal)$
sends $T_1$ to $T_2$. Assume $T_1$ is a BBD-induction with respect
to a recollement
\begin{equation}\nonumber
\begin{xymatrix}{\mathcal{Y}\ar[r]^{i_*}&\mathcal{D}\ar@<3ex>[l]_{i^!}\ar@<-3ex>[l]_{i^*}\ar[r]^{j^*}&\mathcal{\mathcal{X}}\ar@<3ex>_{j_*}[l]\ar@<-3ex>_{j_!}[l]}.
\end{xymatrix}
\end{equation}
This means $T_1$ is compatible with the recollement. By Lemma
\ref{newreco}, the autoequivalence $\Phi$ induces a new recollement
\begin{equation}\nonumber
\begin{xymatrix}{\mathcal{Y}\ar[r]^{\Phi i_*}&\mathcal{D}\ar@<3ex>[l]_{i^!\Psi}\ar@<-3ex>[l]_{i^*\Psi}\ar[r]^{j^*\Psi}&\mathcal{\mathcal{X}}\ar@<3ex>_{\Phi j_*}[l]\ar@<-3ex>_{\Phi j_!}[l]}.
\end{xymatrix}
\end{equation}
where $\Psi$ is the quasi-inverse of $\Phi$. It is straightforward
to check that $T_2$ is compatible with this new recollement. The
lemma follows now from Corollary \ref{Ind-Res}.
\end{proof}

As mentioned in the preliminaries section, bounded t-structures
whose heart is a length category with finitely many simple objects
are of particular interest when dealing with derived categories of
finite dimensional algebras. A useful relation between the hearts of t-structures in $\Xcal$ and $\Ycal$ and the heart of the t-structure in $\Dcal$ is established in \cite{BBD}.

\begin{proposition}[\cite{BBD} Proposition 1.4.18]\label{abelian quotient}
Let $\mathcal{A}_\mathcal{X}$ and $\mathcal{A}_\mathcal{Y}$ be the
hearts of $T_{\mathcal{X}}$ and $T_{\mathcal{Y}}$, repectively, for
a pair $(T_{\mathcal{X}},T_{\mathcal{Y}})
\in\mathcal{T}_\mathcal{X}\times\mathcal{T}_\mathcal{Y}$. Also, let
$\mathcal{A}$ be the heart of
$\Ind(T_{\mathcal{X}},T_{\mathcal{Y}})$. Then $\Acal_\Xcal\cong \Acal/i_*\Acal_\Ycal$ as abelian categories.
\end{proposition}

\begin{remark}
Note that, in the notation of the proposition,
$j^*(\mathcal{A})\subset \mathcal{A}_\mathcal{X}$ and
$i_*(\mathcal{A}_\mathcal{Y})\subset \mathcal{A}$. This follows from
the definitions of restriction and induction and Proposition
\ref{Ind-Res}. Moreover, if we identify $\mathcal{Y}$ with its image
by $i_*$, then we get $\mathcal{A}_\mathcal{Y}\cong \mathcal{A}\
\cap$ Im$(i_*)$.
\end{remark}

The following result will be used later.

\begin{proposition} \label{lengthheart-ind}
Let $\mathcal{A}_\mathcal{X}$ and $\mathcal{A}_\mathcal{Y}$ be the
hearts of $T_{\mathcal{X}}$ and $T_{\mathcal{Y}}$, repectively, for
a pair $(T_{\mathcal{X}},T_{\mathcal{Y}})
\in\mathcal{T}_\mathcal{X}\times\mathcal{T}_\mathcal{Y}$. Also, let
$\mathcal{A}$ be the heart of
$\Ind(T_{\mathcal{X}},T_{\mathcal{Y}})$. Then
$\mathcal{A}_\mathcal{X}$ and $\mathcal{A}_\mathcal{Y}$  are length
categories if and only if so is $\mathcal{A}$.
\end{proposition}

\begin{proof}
Suppose $\mathcal{A}_\mathcal{X}$ and $\mathcal{A}_\mathcal{Y}$ are length hearts and
let $Z\in\mathcal{A}$ be an object of infinite length, i.e., there
exists an infinite sequence of strict monomorphisms of $\mathcal{A}$
\begin{equation}\nonumber
\begin{xymatrix}{... \ar[r]^{f_{k-2}}& Z_{k-1}\ar[r]^{f_{k-1}}& Z_k \ar[r]^{f_k}& Z_{k+1}\ar[r]^{f_{k+1}}& ... \ar[r]& Z,}\end{xymatrix}
\end{equation}
where $k$ runs over an index set $K\subset\mathbb{Z}$ of consecutive
integers, whose image by $j^*$ is a chain of monomorphisms
\begin{equation}\nonumber
\begin{xymatrix}{... \ar[r]^{j^*f_{k-2}\ \ }& j^*Z_{k-1}\ar[r]^{j^*f_{k-1}\ }& j^*Z_k \ar[r]^{j^*f_k\ \ }& j^*Z_{k+1}\ar[r]^{j^*f_{k+1}}& ... \ar[r]& j^*Z}\end{xymatrix}
\end{equation}
that cannot be strict since $\mathcal{A}_\mathcal{X}$ is a length
category. Since $j^*$ is exact, this means that $j^*X_k\neq 0$ for
only finitely many $k\in K$, where $X_k:=$coker$(f_k)$. Let us say
that $j^*X_k=0$ for all $k\in K$ such that $k\leq a$ and $k\geq b$
for some $a,b\in K$. Hence, $K_{\leq a}:=\left\{k\in K: k\leq
a\right\}$ is infinite or $K_{\geq b}:=\left\{k\in K: k\geq
b\right\}$ is infinite. Suppose $K_{\leq a}$ is infinite and
consider the objects $Y_l:=$coker$(Z_{a-l}\rightarrow Z_a)$, for
$l\in\mathbb{N}$, where the map is the composition
$f_{a-l}f_{a-l+1}...f_{a-1}$. Clealry $j^*(Y_l)=0$ (since
$j^*(Z_{a-l})\cong j^*(Z_a)$) and hence $Y_l\in\mathcal{H}\ \cap$
Im$(i_*)$. But this means we have a strict chain of epimorphisms
\begin{equation}\nonumber
...\rightarrow Y_{l+1} \rightarrow Y_{l}\rightarrow
Y_{l-1}\rightarrow ... \rightarrow Y_1
\end{equation}
which is a contradiction since
$\mathcal{A}_\mathcal{Y}\cong\mathcal{A}\ \cap$ Im$(i_*)$ is a
length category. Similarly, if $K_{\geq b}$ is infinite, then one
can define $Y_l:=$coker$(Z_b\rightarrow Z_{b+l})$ and the same
argument will hold.

Conversely, if $\Acal$ is a length heart, then it is easy to check that so is every full subcategory and every quotient category of $\Acal$. Then, by Proposition \ref{abelian quotient}  $\Acal_\Ycal$ and $\Acal_\Xcal$ are length categories.
\end{proof}

\begin{remark}
We can, moreover, say something about the simple objects of the induced heart, as stated in \cite[Proposition 1.4.26]{BBD}. In the notation of the above proposition, it is easy to see that the adjunction morphisms given by the recollement induce a natural transformation of functors from $\Acal_\Xcal$ to $\Acal$
\begin{equation}\nonumber
\phi: H^0j_!\epsilon \rightarrow H^0j_*\epsilon
\end{equation}
where $H^0$ is the cohomological functor associated to the induced t-structure $T$ and $\epsilon$ is the natural embedding of $\Acal_\Xcal$ in $\Xcal$. Define the functor $j_{!*}$ from $\Acal_\Xcal$ to $\Acal$ by setting $j_{!*}(X):=$Im $\phi(X)$. Then it can be proved that the simple objects of $\Acal$ are of the form $i_*Y$ for $Y$ simple in $\Ycal$ or of the form $j_{!*}X$ for $X$ simple in $\Xcal$.
\end{remark}

\end{section}

\begin{section}{t-structures for semisimple algebras}

In this section we classify t-structures for semisimple algebras.

Let $\mathcal{D}_1$ and $\mathcal{D}_2$ be two triangulated
categories with t-structures $T_1 = (\Dcal_1^{\leq 0}, \Dcal_1^{\geq
0})$ and $T_2 = (\Dcal_2^{\leq 0}, \Dcal_2^{\geq 0})$ respectively.
It is clear that the direct sum $T_1 \oplus T_2$, defined to be
$(\Dcal_1^{\leq 0} \oplus \Dcal_2^{\leq 0}, \Dcal_1^{\geq 0} \oplus
\Dcal_2^{\geq 0})$, is a t-structure in $\Dcal_1 \oplus \Dcal_2$.
Conversely all t-structures in $\Dcal_1 \oplus \Dcal_2$ arise in
this way.

\begin{lemma} Every t-structure in the direct sum $\Dcal_1 \oplus
\Dcal_2$ is of the form
\begin{equation}\nonumber
(\Dcal_1^{\leq 0} \oplus \Dcal_2^{\leq 0},
\Dcal_1^{\geq 0} \oplus \Dcal_2^{\geq 0}),
\end{equation}
where $(\Dcal_1^{\leq
0}, \Dcal_1^{\geq 0})$ and $(\Dcal_2^{\leq 0}, \Dcal_2^{\geq 0})$
are t-structures in $\Dcal_1$ and $\Dcal_2$ respectively.
\end{lemma}

\begin{proof} Consider the recollement
\begin{equation}\nonumber
\begin{xymatrix}{\mathcal{D}_1\ar[r]^{i_*}&\mathcal{D} \ar@<3ex>[l]_{i^!}\ar@<-3ex>[l]_{i^*}\ar[r]^{j^*}&\mathcal{D}_2\ar@<3ex>_{j_*}[l]\ar@<-3ex>_{j_!}[l]}.
\end{xymatrix}
\end{equation}\\
where $\Dcal = \Dcal_1 \oplus \Dcal_2$, $i_*:\Dcal_1 \ra \Dcal$ and
$j_! = j_*: \Dcal_2 \ra \Dcal$ are canonical full embeddings, and
$i^*=i^!: \Dcal \ra \Dcal_1$ and $j^*: \Dcal \ra \Dcal_2$ are
canonical projections. The BBD-induction of t-structures $T_1$ in
$\Dcal_1$ and $T_2$ in $\Dcal_2$ is exactly the direct sum $T_1
\oplus T_2$. Conversely, by Corollary \ref{Ind-Res}, to show a
t-structure $T = (\Dcal^{\leq 0}, \Dcal^{\geq 0})$ in $\Dcal$ arises
in this way, it suffices to show it is compatible with the given
recollement, namely that $j_!j^*\Dcal^{\leq 0} \subset \Dcal^{\leq 0}$. Now,
given an object $X \in \Dcal^{\leq 0}$, it is a direct sum of $X_1 \oplus X_2$
for $X_i \in \Dcal_i$ ($i=1,2$). The aisle $\Dcal^{\leq 0}$ is the left perpendicular category  of $\Dcal^{\geq
1}$ (with respect to the pairing Hom$_{\Dcal}(-,-)$) and thus $\Dcal^{\leq 0}$ is closed under taking direct summand.
Hence $j_!j^*(X) = j_! (X_2) = X_2$ belongs to $\Dcal^{\leq 0}$, as wanted.
\end{proof}

Let $\mathbb{K}$ be any field. The Artin-Wedderburn's theorem says finite
dimensional semisimple $\mathbb{K}$-algebras are built up by matrix
algebras over division rings. In other words they are the of the
form $M_{n_1}(\D_1) \times \ldots \times M_{n_s}(\D_s)$, where
$\D_i$ is a division ring over $\mathbb{K}$ and $n_i$ a natural
numer (for all $i$).

\begin{lemma}\label{div ring}
Let $\D$ be a division ring. Then a t-structure in
$\mathcal{D}^b(\mathbb{D})$ is either trivial (i.e., it has a trivial
aisle or a trivial coaisle) or it is a shift of the standard
t-structure.
\end{lemma}

\begin{proof}
The only indecomposable objects of $\mathcal{D}^b(\mathbb{D})$ are
$\mathbb{D}[p]$ for $p\in\mathbb{Z}$. If the t-structure is
non-trivial, then there is a minimal $p_0$ such that
$\mathbb{D}[p_0]\in \mathcal{D}^{\leq 0}$  (and thus
$\mathbb{D}[p]\in \mathcal{D}^{\leq 0}$ for all $p\geq p_0$).
Clearly no other indecomposable can lie in $\mathcal{D}^{\leq 0}$ by
minimality of $p_0$ and thus $\mathcal{D}^{\leq 0}$ coincides with
$\mathcal{D}_0^{\leq 0}[-p_0]$, where $\mathcal{D}_0^{\leq 0}$ is
the standard aisle for $\mathcal{D}^b(\mathbb{D})$.
\end{proof}

We identify $\mathcal{T}_{\mathcal{D}^b(\mathbb{D})}$ with the set
$\bar{\mathbb{Z}}:=\mathbb{Z}\cup\left\{-\infty,+\infty\right\}$
where an integer $n\in\mathbb{Z}$ represents the t-structure
associated with the aisle $\mathcal{D}_0^{\leq n}$, $-\infty$
represents the trivial t-structure given by $\mathcal{D}^{\leq 0}=0$
and $+\infty$ represents the trivial t-structure given by
$\mathcal{D}^{\leq 0}=\mathcal{D}^b(\D)$. Those t-structures indexed
by $n\in\Z$ are bounded, and they are equivalent by shifts in the
derived category $\Dcal^b(\D)$. So, up to derived autoequivalences, there are only
three t-structures in $\Tcal_{\Dcal^b(\D)}$ and there is only one among those that is bounded.

\begin{corollary} Let $A = M_{n_1}(\D_1) \times \ldots \times M_{n_s}(\D_s)$
be a semisimple algebra over $\K$ with $s$ blocks. Then the
t-structures in the derived category $\Dcal^b(A)$ of $A$ is indexed by
$\bar{\Z} \times \ldots \times \bar{\Z}$ ($s$ times). Among them
those indexed by $\Z\times \ldots \times \Z$ ($s$ times) are
bounded.
\end{corollary}

\begin{remark}
A similar observation to the above allows us to conclude that, up to derived autoequivalences, there is only one bounded t-structure. Indeed, for any n-tuple of integers, there is an autoequivalence of the derived category taking the correspondent t-structure to the standard one - and this autoequivalence is just built from a suitable choice of triangulated shifts in each component.
\end{remark}

\end{section}

\begin{section}{Bounded t-structures in $A_n$}
In this section we prove that any bounded t-structure in
$\mathcal{D}^b(A_n)$, where $A_n$ is the path algebra of
\begin{equation}\nonumber
\begin{xymatrix}{\bullet_1\ar[r]&\bullet_2\ar[r]&...\ar[r]&\bullet_{n-1}\ar[r]&\bullet_n}\end{xymatrix}
\end{equation}
over a field $\mathbb{K}$, is BBD-induced with respect to the
recollement associated with an idempotent $e_r$ ($r=1,2,\ldots,n$).
Recall that given an idempotent $e$ on a hereditary algebra $A$ we have
a recollement $\Rcal_e$ (sometimes denoted $\Rcal_r$ if $A$ is a path algebra and $e=e_r$) of the form(\cite{PS, ALK1}):
\begin{equation}\nonumber
\begin{xymatrix}{\mathcal{D}^b(A/AeA)\ar[r]^{\ \ \ i_*}&\mathcal{D}^b(A)\ar@<3ex>[l]_{\ \ \ i^!}\ar@<-3ex>[l]_{\ \ \ i^*}\ar[r]^{j^*}&\mathcal{D}^b(eAe)\ar@<3ex>_{j_*}[l]\ar@<-3ex>_{j_!}[l]}.
\end{xymatrix}
\end{equation}
where the functors can all be explicitly described as derived
functors.  For our purposes we just need
\begin{equation}\nonumber
j^*=-\otimes_A^L Ae, \ \ \
j_!=-\otimes_{eAe}^L A.
\end{equation}

\begin{theorem}
For any bounded t-structure $T=(\mathcal{D}^{\leq 0},\mathcal{D}^{\geq 0})$
in $\Dcal^b(A_n)$, there is an idempotent $e_r$ ($1\leq r\leq n$) such that $T$ is compatible
with $\mathcal{R}_r$.
\end{theorem}
\begin{proof}
We start with a simple homological computation. Since the algebra is
hereditary we represent any indecomposable object by its projective
resolution. Note that in such a projective resolution each
projective has atmost multiplicity one. Also, we use the convention
$P_0=0$. For an idempotent $r$, denote the functors in
$\mathcal{R}_r$ by $j_{!}^r, j^*_r$. We observe that
for $1\leq k\leq n$, $0\leq l\leq n-1$ we have
\begin{equation}\nonumber
j_{!}^rj^*_r(P_l\rightarrow P_k) =   \left\{ \begin{array}{ll}
        0\ \  {\rm if}\  r\leq l\ {\rm or}\ r>k \\
        P_r\ \ {\rm if}\ l<r\leq k
    \end{array} \right.
\end{equation}

Let $T=(\mathcal{D}^{\leq
0},\mathcal{D}^{\geq 0})$ be a bounded t-structure in $\mathcal{D}^b(A_n)$. By lemma \ref{contains standard aisle},
$\mathcal{D}^{\leq 0}$ contains a shift of the standard
aisle. Without loss of generality, suppose that
$\mathcal{D}_0^{\leq -1}\subset\mathcal{D}^{\leq 0}$ and
$\mathcal{D}_0^{\leq 0}\not\subset\mathcal{D}^{\leq 0}$. Now let
$\mathcal{S}$ be the finite set of indecomposable objects lying in
$\mathcal{D}^{\leq 0}\cap {\rm mod}(A_n)$ where ${\rm mod}(A_n)$ is
identified with the heart of the standard t-structure. It is clear that for a negative integer $k$, the intersection $\Dcal^{\leq 0}\cap ({\rm mod} (A_n)[k])$ is contained in the $k$-th shift $\mathcal{S}[k]$ of $\mathcal{S}$. If $\mathcal{S}=\emptyset$ then the t-structure is a shift of the standard t-structure and it is induced 
with respect to any recollement $\mathcal{R}_r$. Suppose $\mathcal{S}\neq\emptyset$. Define
\begin{equation}\nonumber
\begin{array}{c}
m_0:= {\rm max}\left\{k: P_l\rightarrow P_k \in \mathcal{S}\right\} \\
m_1:= {\rm min}\left\{l: P_l\rightarrow P_{k} \in
\mathcal{S}\right\}.
\end{array}
\end{equation}
Note that $1\leq m_0\leq n$ and $0\leq m_1\leq n-1$. We consider the following three cases:
\begin{enumerate}
\item $(m_0,m_1)\neq (n,0)$;  
\item $(m_0,m_1)=(n,0)$ and $P_n\in\mathcal{S}$;
\item $(m_0,m_1)=(n,0)$ and $P_n\notin\mathcal{S}$;
\end{enumerate}
and we prove, case by case, the existence of a vertex $r$ such that the given t-structure is
compatible with $\mathcal{R}_r$, i.e., the aisle $\Dcal^{\leq 0}$ is closed under the functor $j_!^rj^*_r$.

Case (1): Suppose $m_0\neq n$ and we choose $r>m_0$ (e.g. $r=n$). For
$P_l\rightarrow P_k \in \mathcal{S}$ it holds $k\leq m_0< r$. Hence the image of
$\mathcal{S}$ by the functor $j^{r}_!j^*_r$ is zero. Since $\Dcal^{\leq -1}_0$ is always closed under the functor $j_!^rj^*_r$, the compatibility condition holds. Suppose $m_1 \neq 0$ then we choose $r\leq m_1$ (e.g. $r=m_1$). It follows also that $\mathcal{S} \subset \Ker(j_!^rj^*_r)$, and hence the compatibility condition holds.

Case (2): Suppose now $(m_0,m_1)=(n,0)$ and $P_n\in\mathcal{S}$. Therefore, $P_s\rightarrow P_n\in\mathcal{S}$ for all $0\leq s\leq n-1$. If, for any negative integer $k\in\Z_{<0}$, whenever we have $(P_s\ra P_n)[k]$ for some $1\leq s\leq n-1$ in the aisle we also have $P_n[k]$ in the aisle, then the idempotent $r=n$ fulfills the compatibility condition. Assume now that for some negative integer $k$, $P_n[k]$ does not belong to the aisle $\Dcal^{\leq 0}$ but both $P_n[k+1]$ and $(P_s\ra P_n)[k]$ (for some $1\leq s\leq n-1$) do. We take the minimal $s_0$ such that $(P_{s_0} \ra P_n)[k]$ belongs to $\Dcal^{\leq 0}$. From the triangle
\[(P_{s_0} \ra P_n)[k] \ra P_{s_0}[k+1] \ra P_n[k+1] \ra (P_{s_0}\ra P_n)[k+1]\]
we see that $P_{s_0}[k+1]$ belongs to $\Dcal^{\leq 0}$. For $0\leq t <s_0$ and $s_0\leq m\leq n$, the triangle
\[(P_t\ra P_m)[k] \ra (P_t\ra P_n)[k] \oplus (P_{s_0}\ra P_m)[k] \ra (P_{s_0}\ra P_n)[k] \ra (P_t\ra P_m)[k+1],\]
where $P_{s_0}\ra P_m$ is viewed as zero when $s_0=m$, shows that $(P_t\ra P_m)[k]$ does not belong to $\Dcal^{\leq 0}$ (because of the minimality of $s_0$). Note that $P_t\ra P_m$ for $1\leq t <s_0 \leq m$ are all the indecomposable modules which not killed by $j_!^{s_0}j^*_{s_0}$ and are sent to $P_{s_0}$. Therefore the idempotent $r=s_0$ is as desired.

Case (3): Suppose $(m_0,m_1)=(n,0)$ and $P_n\not\in\mathcal{S}$. Let
\begin{equation}\nonumber
p_0:={\rm max}\left\{k:P_k\in\mathcal{S}\right\}\ \ {\rm and}\ \
p_1:={\rm min}\left\{l:P_l\rightarrow P_n\in\mathcal{S}\right\}
\end{equation}
and note that for all $l\geq p_1$, $P_l\rightarrow
P_k\in\mathcal{S}$. This implies that $p_0< p_1$ as otherwise the
triangle
\begin{equation}\nonumber
P_{p_0}\longrightarrow P_n \longrightarrow (P_{p_0}\rightarrow P_n)
\longrightarrow P_{p_0}[1]
\end{equation}
would imply $P_n\in\mathcal{S}$. Choose $r=p_0+1$. If we show that
$\mathcal{S}\subset$ Im$(i_*^r)$, then clearly $\mathcal{S}\subset$
Ker$(j^*_r)$ and thus we get the compatibility condition as for case (1). Now,
$P_l\rightarrow P_k\in$ Im$(i_*^r)$ if and only if $l\geq r$ or $k<
r$. Suppose $P_l\rightarrow P_k\in\mathcal{S}$ such that $0<l<r$ and
$k\geq r$.
As above, since $1\leq l\leq p_0\leq k-1$, there is a triangle
\begin{equation}\nonumber
P_{p_0}\longrightarrow P_k\oplus(P_l\rightarrow P_{p_0})
\longrightarrow (P_l\rightarrow P_k)\longrightarrow P_{p_0}[1]
\end{equation}
implying that $P_k\in\mathcal{S}$ - a contradiction with
the choice of $p_0$. Thus $\mathcal{S}\subset$ Im$(i_*^r)$.
\end{proof}

The factors appearing in such recollements $\mathcal{R}_r$ for $\Dcal^b(A_n)$
are either $\mathcal{D}^b(\mathbb{K})$, $\mathcal{D}^b(A_k)$ for
some $k<n$ or $\mathcal{D}^b(A_k\times A_p)$ for some $k$ and $p$ with $k+p<n$.
This theorem shows that we can inductively construct all bounded
t-structures in $\Dcal^b(A_n)$. Furthermore we only need the recollements
coming from these idempotents.

As a corollary one can list all bounded t-structures for $\Dcal^b(A_2)$.
Given the simplicity of computations for $A_2$ we can, indeed, go
one step further and describe all t-structures, without the boundedness constraint.  
Recall that for a triangulated category $\Dcal$ endowed with a recollement $\mathcal{R}$, we write $\mathcal{T}^{\mathcal{R}}_{\mathcal{D}}$ for the set of t-structures in $\Dcal$ which are compatible with the recollement $\mathcal{R}$.

\begin{proposition}
If $\mathcal{D}^{\leq 0}\subset \mathcal{D}^b(A_2)$ is an aisle of a
t-structure in $\Tcal^{\mathcal{R}_1}_{\Dcal^b(A_2)}\cup \Tcal^{\mathcal{R}_2}_{\Dcal^b(A_2)}$, then it
can be described by one of the following types:
\begin{enumerate}
\item It is trivial (i.e., equal to $0$ or to $\mathcal{D}^b(A_2)$);
\item It is a shift of the standard aisle;
\item Its indecomposable objects are (some or all) shifts of $P_1$;
\item Its indecomposable objects are (some or all) shifts of $S_2$;
\item Its indecomposable objects are (some or all) shifts of $P_2$;
\item Its indecomposable objects are the union of the indecomposable objects of a shift of a standard aisle with (some or all) shifts of $P_1$;
\item Its indecomposable objects are the union of the indecomposable objects of a shift of a standard aisle with (some or all) shifts of $S_2$;
\item Its indecomposable objects are the union of the indecomposable objects of a shift of a standard aisle with (some or all) shifts of $P_2$ and exactly one extra shift of $S_2$.
\end{enumerate}
Furthermore these are all the aisles of $\mathcal{D}^b(A_2)$.
\end{proposition}
\begin{proof}
Let $A$ be the path algebra of the quiver $A_2$.
Note that, as rings, $A/Ae_rA\cong\mathbb{K}\cong e_rAe_r$ for $r=1,2$.  By Lemma \ref{div ring} $\mathcal{T}_{\Dcal^b(\mathbb{K})} \cong \bar{\mathbb{Z}}$. We check first that
these t-structures can occur as BBD-induction from pairs. We will use the notation
$\mathcal{D}^{\leq 0}_{(r;m,n)}$ to describe the aisle resulting
from induction of the pair  $(m,n)\in \bar{\mathbb{Z}}\times
\bar{\mathbb{Z}}\cong \mathcal{T}_{\mathcal{D}^b(A/Ae_rA)}\times
\mathcal{T}_{\mathcal{D}^b(e_rAe_r)}$ with respect to the recollement
$\mathcal{R}_r$ ($r=1,2$). The following statements are easy to check using the explicit functors in $\Rcal_r$:
\begin{enumerate}
\item If $\Dcal^{\leq 0}$ is trivial (i.e., equal to $0$ or to $\mathcal{D}^b(A_2)$) then we have that $\Dcal^{\leq 0}$ is equal to either  $\mathcal{D}^{\leq 0}_{(r;-\infty,-\infty)}$ or $\mathcal{D}^{\leq 0}_{(r;+\infty,+\infty)}$;
\item If $\Dcal^{\leq 0}$ is a shift of the standard aisle, it is equal to $\mathcal{D}^{\leq 0}_{(r;n,n)}$;
\item If the indecomposable objects of $\Dcal^{\leq 0}$ are shifts of $P_1$ then we have that $\Dcal^{\leq 0}$ is equal to either $\mathcal{D}^{\leq 0}_{(1;+\infty,-\infty)}=\mathcal{D}^{\leq 0}_{(2;-\infty,+\infty)}$ (if all shifts are there) or $\mathcal{D}^{\leq 0}_{(1;n,-\infty)}=\mathcal{D}^{\leq 0}_{(2;-\infty,n)}$ (if only some shifts are there);
\item If  the indecomposable objects of $\Dcal^{\leq 0}$ are shifts of $S_2$ then we have that $\Dcal^{\leq 0}$ is equal to either $\mathcal{D}^{\leq 0}_{(1;-\infty,+\infty)}$ (if all shifts are there) or $\mathcal{D}^{\leq 0}_{(1;-\infty,n)}$ (if only some shifts are there);
\item If the indecomposable objects of $\Dcal^{\leq 0}$ shifts of $P_2$ then we have that $\Dcal^{\leq 0}$ is equal to either $\mathcal{D}^{\leq 0}_{(2;+\infty,-\infty)}$ (if all shifts are there) or $\mathcal{D}^{\leq 0}_{(2;n,-\infty)}$ (if only some shifts are there);
\item If the indecomposable objects of $\Dcal^{\leq 0}$ are the union of the indecomposable objects of a shift of a standard aisle with shifts of $P_1$ then we have that $\Dcal^{\leq 0}$ is equal to either  $\mathcal{D}^{\leq 0}_{(1;+\infty,n)}=\mathcal{D}^{\leq 0}_{(2;n,+\infty)}$ (if all shifts are there) or $\mathcal{D}^{\leq 0}_{(1;n,m)}=\mathcal{D}^{\leq 0}_{(2;m,n)}$ for $m<n$ (if only some shifts are there);
\item If the indecomposable objects of $\Dcal^{\leq 0}$ are the union of the indecomposable objects of a shift of a standard aisle with shifts of $S_2$ then we have that $\Dcal^{\leq 0}$ is equal to either  $\mathcal{D}^{\leq 0}_{(1;m,+\infty)}$ (if all shifts are there) or $\mathcal{D}^{\leq 0}_{(1;m,n)}$ for $m<n$ (if only some shifts are there);
\item If the indecomposable objects of $\Dcal^{\leq 0}$ are the union of the indecomposable objects of a shift of a standard aisle with shifts of $P_2$ and exactly one extra shift of $S_2$ then we have that $\Dcal^{\leq 0}$ is equal to either  $\mathcal{D}^{\leq 0}_{(2;+\infty,n)}$ (if all shifts are there) or $\mathcal{D}^{\leq 0}_{(2;m,n)}$ for $m>n$ (if only some shifts are there).
\end{enumerate}

We now prove that these are all possible aisles.
Suppose $\mathcal{D}^{\leq 0}$ is neither zero nor the whole
category (i.e. not of type (1)). Then it contains some
indecomposable objects. If the aisle contains shifts of only one
indecomposable module, then all possibilities are listed in the
proposition - they are types (3),(4) and (5). So we suppose there
are at least two indecomposable objects which are shifts of distinct
indecomposable modules. This implies that a standard aisle is contained in $\Dcal^{\leq0}$, since we will have at least one triangle of indecomposable objects in $\Dcal^{\leq0}$ which is then necessarily a shift of
\begin{equation}\nonumber
P_1\longrightarrow P_2 \longrightarrow S_2 \longrightarrow P_1[1].
\end{equation}
We assume, without loss of generality,
that $\mathcal{D}^{\leq -1}_0\subset \mathcal{D}^{\leq 0}$ but
$\mathcal{D}^{\leq 0}_0\not\subset \mathcal{D}^{\leq 0}$. If
$\mathcal{D}^{\leq -1}_0= \mathcal{D}^{\leq 0}$ we are done (type
(2)), otherwise, by considering the triangle above we observe the following:
\begin{itemize}
\item  if $S_2\in \mathcal{D}^{\leq 0}$ then $P_1\not\in \mathcal{D}^{\leq 0}$ (otherwise $S_2\in \mathcal{D}^{\leq 0}$ and thus $\mathcal{D}^{\leq 0}_0= \mathcal{D}^{\leq 0}$). If $P_2\not\in \mathcal{D}^{\leq 0}$ then the only remaining indecomposable objects that can lie in the aisle are the negative shifts of $S_2$ and this is type (7).  If $P_2\in \mathcal{D}^{\leq 0}$ then $S_2[-1]\not\in \mathcal{D}^{\leq 0}$ (otherwise  $P_1\in \mathcal{D}^{\leq 0}$ and thus $\mathcal{D}^{\leq 0}_0= \mathcal{D}^{\leq 0}$) and thus the only remaining indecomposable objects that can lie in the aisle are negative shifts of $P_2$ and this is type (8);
\item if $P_2\in \mathcal{D}^{\leq 0}$, then so is $S_2$ and we fall on the previous case;
\item If $P_1\in \mathcal{D}^{\leq 0}$, then $P_2, S_2\not\in \mathcal{D}^{\leq 0}$ and the only remaining indecomposable objects that can lie in the aisle are the negative shifts of $P_1$ - and this is type (6).
\end{itemize}
This concludes the proof.
\end{proof}

We end this section by two clarifying remarks making use of the explicit simple nature of $\Dcal^b(A_2)$.

\begin{remark}
The proof above shows how different recollements allow different types of
induced t-structures. While $\mathcal{R}_1$ and $\mathcal{R}_2$ both
allow types (1), (2), (3), and (6) to appear as BBD-induction, only
$\mathcal{R}_1$ induces types (4) and (7) and only $\mathcal{R}_2$
induces types (5) and (8).
\end{remark}

\begin{remark}
Also it is not hard to check the equivalence classes for bounded t-structures in $\mathcal{D}^b(A_2)$. Indeed, by the work of Miyachi and Yekutieli (\cite{MY}), the group of autoequivalences of $A_2$ is known to be generated by the Auslander-Reiten translation $\tau$ and the triangulated shift $[1]$. By explicitly computing the orbits of a bounded t-structure in the list above, one can easily see that the equivalence class of a bounded t-structure $(\mathcal{D}^{\leq 0}, \mathcal{D}^{\geq 0})$ is determined by the number of connected components of the intersection of the Auslander-Reiten quiver of $\mathcal{D}^b(A_2)$ with $\mathcal{D}^{\leq 0}$. Therefore, the set of equivalence classes of bounded t-structures for  $\mathcal{D}^b(A_2)$ is naturally parametrized by $\mathbb{N}$.
\end{remark}

\end{section}

\begin{section}{Bounded t-structures for piecewise hereditary algebras}

In this section we generalize our result of the previous section to
piecewise hereditary algebras, and prove that any bounded t-structure with length heart is a BBD-induction with respect to some recollement by derived categories. Moreover when the algebra is hereditary of finite representation type, we prove that any bounded t-structure has a length heart, and hence it is always a BBD-induction. Indeed, since we are only interested in recollements by derived categories (and not general triangulated categories, see \cite{ALK2} for a discussion), by the expression \textit{a t-structure is BBD-induced} we will mean \textit{a t-structure is BBD-induced with respect to some recollement by derived categories}.

Let $A$ be a finite dimensional $\mathbb{K}$-algebra. Recall that
$A$ is {\em piecewise hereditary}, if there exists a hereditary and
abelian $\K$-category $\mathcal{H}$ such that the bounded derived
categories $\Dcal^b(A)$ and $D^b(\mathcal{H})$ are triangule
equivalent. In other words, there exists a tilting complex $T$ in
$D^b(\mathcal{H})$ with endomorphism ring being $A$. In particular,
the derived category $\Dcal^b(A)$ is hereditary, i.e., each
indecomposable object is a stalk complex concentrated in one
component, or equivalently a shift of an indecomposable object in
$\mathcal{H}$.

Given a family of simple-minded objects $X_1, \ldots, X_n$ in
$\Dcal^b(A)$, by Subsection \ref{simple-minded}, their extension
closure is a module category of a certain finite dimensional algebra
$\Gamma$. In the piecewise hereditary case, this algebra $\Gamma$ is
even directed, in the sense that the quiver of $\Gamma$ has no
oriented cycles, or equivalently, there are no cycles
\begin{equation}\nonumber
P_1\ra P_2\ra \ldots \ra P_s = P_1
\end{equation}
of nonzero homomorphisms (which are not isomorphisms) between projective indecomposable $\Gamma$-modules.
This follows from Happel \cite[Lemma IV.1.10]{H}. For the
convenience of the reader we include a proof here.

\begin{lemma}[\cite{H} Lemma IV.1.10] \label{directness} 
Let $A$ be a piecewise hereditary algebra over $\K$, and $X_1,
\ldots, X_n$ a family of simple-minded objects in $\Dcal^b(A)$. Then
the extension closure of $X_i$'s in $\Dcal^b(A)$ is equivalent to
$\fdmod(\Gamma)$ for some finite dimensional directed algebra
$\Gamma$. In particular, $X_i$ is exceptional (i.e. has no
self-extensions).
\end{lemma}

\begin{proof} Let $\Hcal$ be an abelian and hereditary category $\mathcal{H}$
with $\Dcal^b(A) \cong \Dcal^b(\mathcal{H})$. By Happel and Reiten
\cite{HR} and Lenzing \cite{L}, $\mathcal{H}$ can be chosen to be
either the module category $\fdmod(H)$ of some finite dimensional
hereditary algebra $H$, or $\coh(X)$ the category of coherent
sheaves over some exceptional curve $X$. In these two cases, given
indecomposable objects $X,Y\in\Hcal$ with $\Ext_{\Hcal}^1(Y,X)=0$,
then a nonzero homomorphism from $X$ to $Y$ is either a monomorphism
or an epimorphism (\cite[4.1]{HRi}). In particular if $X$ has no
self-extensions then the endomorphism ring $\End_{\Hcal}(X)$ is a
division ring.

Assume now $P_ 1\ra P_2 \ra \ldots \ra P_s = P_1$ is a cycle of
indecomposable projective $\Gamma$-modules. View this as a cycle in
the derived category $\Dcal^b(A) \cong\Dcal^b(\Hcal)$. Because
$\Hcal$ is hereditary, there exists an integer $k$ such that all
$P_i$'s belong to the $k$-th shift $\mathcal{H}[k]$
(\cite[I.5.3]{H}). Without loss of generality assume $k=0$. For any
$i,j =1,\ldots, s$, by Lemma \ref{extension-in-heart},
$0=\Ext^1_{\Gamma}(P_i,P_j) = \Hom_{\Dcal^b(\Hcal)}(P_i,P_j[1])$,
which is isomorphic to $\Ext^1_{\Hcal}(P_i,P_j)$. If follows that
every map occuring in the cycle is either a monomorphism or
epimorphism. Indeed, either all maps are monomorphisms or all maps
are epimorphisms (otherwise we would get a nonzero map which is a composition of a
proper monomorphism with a proper epimorphism but
neither a monomorphism nor an epimorphism). Since $\End(P_1)$ is a division ring, all maps
must be isomorphisms, contradiction to the definition of a  cycle.

Finally because of the directedness of $\Gamma$ each
$X_i$ has no self-extensions as a $\Gamma$-module. In particular by Lemma
\ref{extension-in-heart}, $\Hom_{\Dcal^b(A)}(X_i,X_i[1]) =
\Ext^1_{\Gamma}(X_i,X_i) =0$. Since $\Dcal^b(A) \cong
\Dcal^b(\Hcal)$ is hereditary, no higher self-extensions are
possible.
\end{proof}

Before we present our main result, we need one further result, showing how to construct a recollement of a piecewise hereditary algebra from an indecomposable and {\em exceptional} object (i.e., an object without self-extensions). 

\begin{proposition}[\cite{ALK2} Theorem 2.5; \cite{AKL3} Proposition 5.8] \label{construct-recollement}
Let $A$ be a piecewise hereditary algebra over a field $\K$. For any indecomposable and exceptional object $X$ in $\Dcal^b(A)$, there exists a recollement of the form
\begin{equation}\nonumber
\begin{xymatrix}{\mathcal{D}^b(B)\ar[r]^{}&\mathcal{D}^b(A)\ar@<1.5ex>[l]_{}\ar@<-1.5ex>[l]_{}\ar[r]^{}&
\mathcal{D}^b(C)\ar@<1.5ex>_{}[l]\ar@<-1.5ex>_{}[l]}.
\end{xymatrix}
\end{equation}
where $B$ and $C$ are again piecewise hereditary algebras with $C=\End_A(X)$ being a division ring over $\K$.
\end{proposition}

By \cite{HR}, a quasi-hereditary algebra is derived equivalent to either a hereditary algebra or a canonical algebra. These two cases are discussed by \cite[Theorem 2.5]{ALK2} and \cite[Proposition 5.8]{AKL3} respectively. We can now prove our main result.

\begin{theorem} \label{piecewisehereditary}
Let $A$ be a piecewise hereditary algebra over a field $\K$ and
suppose $T=(\mathcal{D}^{\leq 0},\mathcal{D}^{\geq 0})$ is a bounded
t-structure in $\mathcal{D}^b(A)$ whose heart is a length category.
Then $T$ is compatible with a recollement of the
form
\begin{equation}\nonumber
\begin{xymatrix}{\mathcal{D}^b(B)\ar[r]^{}&\mathcal{D}^b(A)\ar@<1.5ex>[l]_{}\ar@<-1.5ex>[l]_{}\ar[r]^{}&
\mathcal{D}^b(C)\ar@<1.5ex>_{}[l]\ar@<-1.5ex>_{}[l]}.
\end{xymatrix}
\end{equation}
(and hence a BBD-induction with respect to it) for some piecewise hereditary algebras $B$ and $C$. Moreover, $C$ can always be taken to be a division ring over the base field $\K$.
\end{theorem}

\begin{proof} Write $\Dcal=\Dcal^b(A)$ for short. Let $T$ be a bounded t-structure in $\Dcal$ with heart $\Acal$ being a length category.
Take a set of simple objects $X_1, \ldots, X_n$ in $\Acal$. By
Theorem \ref{bijection}, there is a
finite dimensional algebra $\Gamma$ such that $\Acal \cong
\fdmod(\Gamma)$ and $X_1,\ldots,X_n$ are simple $\Gamma$-modules. Since $A$ is piecewise hereditary, Lemma \ref{directness} shows that, moreover, $\Gamma$ is directed.
Let $\Hcal$ be an abelian and hereditary category with
$\Dcal^b(\Hcal) \cong \Dcal$. There exists integer $l_i\in\Z$ such
that $X_i \in \Hcal[l_i]$ the $l_i$-th shift (for $i =1,\ldots,n$).
Because of the directness of $\Gamma$, we can assume without loss of
generality that $X_1$ is the simple projective $\Gamma$-module with
$l_1$ smallest among simple projective $\Gamma$-modules. We claim
that $l_i\geq l_1$ for all $2\leq i \leq n$. 
When $X_i$ is again simple projective as $\Gamma$-module, because of the choice
of $l_1$ it holds $l_i\geq l_1$. When $X_i$ is not projective, there
must exists some $X_j$, simple projective as $\Gamma$-module, and a
path in $\fdmod(\Gamma)$ from $X_j$ to $X_i$. View this path in the
derived category $\Dcal^b(\Hcal)$ and we see that $l_i\geq l_j$.
Also $l_j\geq l_1$ because of the choice of $l_1$. Hence $l_i\geq
l_1$ for all $2\leq i\leq n$.

Next we show that $X_2,\ldots,X_n$ belong to the right perpendicular
category of $X_1$, that is, $\Hom_{\Dcal}(X_1,X_i[k])=0$ for $2\leq
i\leq n$ and for all integers $k\in\Z$. According to the definition
of simple-minded objects, the Hom-set vanishes whenever $k\leq 0$.
We only have to concern with the case $k>0$. Let $M_i$ be the
indecomposable objects in $\Hcal$ with $X_i = M_i[l_i]$. We have
\begin{eqnarray*}
\Hom_{\Dcal}(X_1,X_i[k]) &=& \Hom_{\Dcal^b(\Hcal)}(M_1[l_1], M_i[l_i+k])\\
&=&\Hom_{\Dcal^b(\Hcal)}(M_1,M_i[l_i-l_1+k])\\
&=&\Ext^{l_i-l_1+k}_{\Hcal}(M_1,M_i),
\end{eqnarray*}
which is zero whenever $l_i-l_1+k \neq 0, 1$. Under the condition
that $l_i\geq l_1$ and $k>0$, the only possible nontrivial case is
$l_i=l_1$ and $k=1$. In this case $\Hom_{\Dcal}(X_1,X_i[1]) =
\Ext^1_{\Gamma}(X_1,X_i)$, by Lemma \ref{extension-in-heart}, which
is zero because $X_1$ is a projective $\Gamma$-module.

Note that $X_1$ is indecomposable and exceptional in $\Dcal^b(A)$.
By Proposition \ref{construct-recollement}, there exists a recollement of
$\Dcal^b(A)$ generated by $X_1$
\begin{equation}\nonumber
\begin{xymatrix}{\mathcal{D}^b(B)\ar[r]^{\ i_*}&\mathcal{D}^b(A)\ar@<3ex>[l]_{\ i^!}\ar@<-3ex>[l]_{\ i^* }\ar[r]^{j^*}&\mathcal{D}^b(C)\ar@<3ex>_{j_*}[l]\ar@<-3ex>_{j_!}[l]}.
\end{xymatrix}
\end{equation}
where $B$ and $C$ are again piecewise hereditary algebras. Indeed
$C=\End_{\Dcal^b(A)}(X_1)$ is a division ring, $\Img(j_!) =
\tria(X_1)=\{X_1[k]:k\in\Z\}$ and $\Img(i_*)=X_1^\perp
:=\{Y\in\Dcal^b(A): \Hom_{\Dcal^b(A)}(X_1,Y[k])=0,\ \forall\
k\in\Z\}$. We will show that the t-structure $T$ is compatible with
this recollement, i.e., that $j_!j^*(\Dcal^{\leq 0}) \subset \Dcal^{\leq
0}$. Corollary \ref{Ind-Res} shows that $T$ is then a BBD-induction.

Since $X_1,\ldots,X_n$ is the set of simple-minded objects
corresponding to $T$, the aisle $\Dcal^{\leq 0}$ is just the
extension closure of non-negative shifts of $X_i$'s. We have showed
that $X_2,\ldots,X_n$ are right perpendicular to $X_1$, and hence
belong to $\Img(i_*)$. They are sent to zero by $j^*$. It follows
that the indecomposable objects of $j_!j^*(\Dcal^{\leq 0})$ are
exactly $X_1[k]$ for $k\geq 0$, which form a subset of the aisle
$\Dcal^{\leq 0}$.
This finishes our proof.
\end{proof}

Combined with Proposition \ref{lengthheart-ind} we obtain the following immediate
following.

\begin{corollary}\label{cor1} Let $A$ be a piecewise hereditary algebra over a
field $\K$, and $T=(\Dcal^{\leq 0},\Dcal^{\geq 0})$ a bounded
t-structure in $\Dcal^b(A)$. Then the heart of $T$ is a length
category if and only if $T$ is BBD-induced from t-structures with
length hearts.
\end{corollary}

\begin{remark}
In \cite{AKL3} a Jordan H\"older theorem for derived categories of piecewise hereditary algebras is proven. If $A$ is piecewise hereditary over a field $\K$, the derived-simple factors of $\Dcal^b(A)$ are of the form $\Dcal^b(\D)$ where $\D$ is a division ring over $\K$. By Lemma \ref{div ring} every bounded t-structure of $\Dcal^b(\D)$ has a length heart. Hence the bounded t-structures of $\Dcal^b(A)$ with length hearts are precisely the ones obtained by BBD-induction.
\end{remark}

\begin{example}\label{coh}
Let $A$ be the path algebra of the Kroenecker quiver $\mathbb{C}$. 
Then it is well known that  $\Dcal^b(A)\cong \Dcal^b(\text{coh}(\mathbb{P}^1))$ and therefore coh$(\mathbb{P}^1)$ is the heart of a bounded t-structure in $\Dcal^b(A)$. It is not, however, a length heart. Then, since $A$ is hereditary, Corollary \ref{cor1} shows that this bounded t-structure cannot be induced with respect to any recollement by derived categories.
\end{example}

For the rest of the section, we consider hereditary algebras of finite representation type.

\begin{lemma} Let $A$ be a hereditary algebra of
finite representation type, and $T$ any bounded
t-structure in $\Dcal^b(A)$. Then the heart of $T$ is a length category.
\label{lengthheart}
\end{lemma}

\begin{proof} Suppose $T$ is a bounded t-structure in $\Dcal^b(A)$
with heart $\Acal$. If $X$ is an indecomposable object in $\Acal$,
then any nonzero shift $X[k]$ ($k\neq 0$) can not lie in $\Acal$ 
since $\Acal[k]\cap\Acal = 0$ by definition of heart.
But indecomposable objects in $\Dcal^b(A)$ are
stalk complexes and $A$ has finite representation type. Hence
$\Acal$ contains only finitely many indecomposable objects. Assume
there exists some object $X$ in $\Acal$ with infinite length. That
means it admits an infinite sequence of monomorphisms $\ldots
X_{-1} \rightarrow X_0 \rightarrow X_1 \rightarrow \ldots \rightarrow
X$. Write each $X_i$ as a direst sum of indecomposable objects. We
will be able to find a cycle of indecomposable objects. But on the
other hand, since $A$ is hereditary of finite representation type,
we know there exists a total order on indecomposable objects in
$\fdmod(A)$, and hence in $\Dcal^b(A)$, such that $\Hom(X_j,X_i)=0$
whenever $j>i$. A contradiction! Therefore the heart $\Acal$ must be
a length category.
\end{proof}

Combined with Theorem \ref{bijection}, we obtain a bijection between
the set of bounded t-structures in $\Dcal^b(A)$ and equivalence
classes of families of simple-minded objects. It follows from Theorem \ref{piecewisehereditary} that any bounded t-structure can be induced from some recollement. In fact the next result shows that up to Auslander-Reiten translation, which is an autoequivalence of $\Dcal^b(A)$, we can choose the recollement to be of type $\mathcal{R}_r$ (i.e. associated with an idempotent $e_r$ of $A$, see the beginning of Section 5).

\begin{corollary}\label{hereditary}
Let $A$ be a hereditary algebra of finite representation type, and $T=(\Dcal^{\leq 0}, \Dcal^{\geq 0})$ any bounded t-structure in $\Dcal^b(A)$. Then there exists an autoequivalence $\Phi$ of $\Dcal^b(A)$ such that $\Phi(T):=(\Phi(\Dcal^{\leq 0}), \Phi(\Dcal^{\geq 0}))$ is compatible with
(and hence a BBD-induction with respect to) a recollement of type $\mathcal{R}_r$ for some $r$.
\end{corollary}

\begin{proof} Let $X_1,\ldots,X_n$ be the family of simple-minded objects associated to the bounded t-structure $T$. First of all we give an alternative way of finding such an $X_1$ as in the proof of Theorem \ref{piecewisehereditary}. For the algebra $A$ is hereditary of finite representation type, as mentioned in the proof of Lemma \ref{lengthheart}, there exists a total order on the set of indecomposable objects in
$\Dcal^b(A)$ such that $X<X[1]$ for any indecomposable $X$, and
$\Hom(X,Y)\neq 0$ implies $X<Y$ for any indecomposable $X$ and $Y$. Without loss of generality we assume that $X_1<X_i$ for all $i \geq 1$. We have to show that $X_i$, for all $i\geq 2$, belongs to the right
perpendicular category $X_1^\perp=\{Y\in\Dcal^b(A):\Hom(X,Y[k])=0,\
\forall\ k\in\Z\}$ of $X_1$. Since $X_1,\ldots,X_n$ is a family of
simple-minded objects, $\Hom(X_1,X_i[k])=0$ for all $i\geq 2$ and
all $k\leq 0$. Assume $\Hom(X_1,X_i[k]) \neq 0$ for some $i\geq 2$
and some $k>0$. Then we have a triangle $X_i[k-1] \ra Y \ra
X_1 \ra X_i[k]$ with $k-1\geq 0$. It follows that $X_i \leq X_i[k-1]
<Y < X_1$, a contradiction with the choice of $X_1$.

Again because $A$ is hereditary of finite representation type, the Auslander-Reiten translation $\tau$ is an autoequivalence of $\Dcal^b(A)$. Moreover, any indecomposable object in $\Dcal^b(A)$ can be transformed to an indecomposable projective $A$-module by iteratively applying $\tau$. So there exists a natural number $s\in\N$ and an idempotent $e_r\in A$ such that $\tau^s(X_1)=P_r$, the indecomposable projective $A$-module associated to $e_r$. It is clear that $\tau^s(T)=(\tau^s(\Dcal^{\leq 0}), \tau^s(\Dcal^{\geq 0}))$ is still a bounded t-structure in $\Dcal^b(A)$, $(\tau^s(X_1),\ldots,\tau^s(X_n))$ is the associated family of simple-minded objects, and $\tau^s(X_i)$, for $i\geq 2$, belong to the right perpendicular category $P_r^\perp$. By Theorem \ref{piecewisehereditary}, $\Phi(T)$ is a BBD-induction with respect to the recollement of $\Dcal^b(A)$ generated by $P_r$. This recollement is precisely of type $\mathcal{R}_r$.
\end{proof}

\end{section}

\end{document}